\documentclass[11pt, reqno]{amsart}
\usepackage{amsmath, enumerate}
\usepackage{amsfonts}
\usepackage{amssymb}
\usepackage{amsthm}
\usepackage{graphicx}
\usepackage{soul}
\usepackage[usenames, dvipsnames]{color} 
\usepackage{hyperref}
\newtheorem{Theorem}{Theorem}[section]
\newtheorem{Corollary}[Theorem]{Corollary}
\newtheorem{Lemma}[Theorem]{Lemma}
\newtheorem{Proposition}[Theorem]{Proposition}

\theoremstyle{definition}

\theoremstyle{remark}

\def \N {\mathbb{N}}
\def \l {\lambda}
\def \O {\Omega}

\def \R {\mathbb{R}}
\def \Rn{\mathbb{R}^n}
\def \K {\mathcal{K}}

\def \e {\epsilon}

\def \a {\alpha}

\def \la {\left\langle }
\def \ra {\right\rangle}

\def \dist{\mathrm{dist}}

\def \tr{\mathrm{tr}}
\def \det{\mathrm{det}}

\def \ra{\right\rangle}
\def \la{\left\langle}

\def\lb{\left\{}
\def\rb{\right\}}

\begin{document}

\title[Inverse Iteration for Monge-Amp\`ere Eigenvalue]{Inverse Iteration for the Monge-Amp\`ere Eigenvalue Problem}
\author[F. Abedin and J. Kitagawa]{Farhan Abedin$^*$ and Jun Kitagawa$^{**}$} 
\address{Department of Mathematics, Michigan State University, East Lansing, MI 48824}
\email{abedinf1@msu.edu}
\address{Department of Mathematics, Michigan State University, East Lansing, MI 48824}
\email{kitagawa@math.msu.edu}
\thanks{$^{**}$JK's research was supported in part by National Science Foundation grant DMS-1700094. \\
\hphantom{11l}$^*$ Corresponding author: abedinf1@msu.edu}
\begin{abstract}
We present an iterative method based on repeatedly inverting the Monge-Amp\`ere operator with Dirichlet boundary condition and prescribed right-hand side on a bounded, convex domain $\O \subset \Rn$. We prove that the iterates $u_k$ generated by this method converge as $k \to \infty$ to a solution of the Monge-Amp\`ere eigenvalue problem
$$\begin{cases}
\det D^2u = \l_{MA} (-u)^n & \quad \text{in } \O,\\
u = 0 & \quad \text{on } \partial \O.
\end{cases}$$
Since the solutions of this problem are unique up to a positive multiplicative constant, the normalized iterates $\hat{u}_k := \frac{u_k}{||u_k||_{L^{\infty}(\O)}}$ converge to the eigenfunction of unit height.
In addition, we show that $\lim\limits_{k \to \infty} R(u_k) = \lim\limits_{k \to \infty} R(\hat{u}_k)  =  \l_{MA}$, where the Rayleigh quotient $R(u)$ is defined as
$$R(u) := \frac{\int_{\O} (-u) \ \det D^2u}{\int_{\O} (-u)^{n+1}}.$$
Our method converges for a wide class of initial choices $u_0$ that can be constructed explicitly, and does not rely on prior knowledge of the Monge-Amp\`ere eigenvalue $\l_{MA}$.
\end{abstract}

\maketitle

\section{Introduction and Main Result}\label{sec:intro}

Let $\O \subset \Rn$ be a bounded, convex domain. The Monge-Amp\`ere eigenvalue problem seeks to find a convex function $u \in C^2(\O) \cap C(\overline{\O})$ and a positive number $\l$ such that
\begin{equation}\label{eigenvalueproblemforMA}
\begin{cases}
\det D^2u = \l (-u)^n & \quad \text{in } \O,\\
u = 0 & \quad \text{on } \partial \O.
\end{cases}
\end{equation}
This problem was first considered by Lions in \cite{Lions85}, who proved the following result. 
\begin{Theorem} [Lions '85]\label{Lions} Assume $\O \subset \Rn$ is a smooth, bounded, uniformly convex domain. There exist a unique positive constant $\l_{MA}$ and a unique (up to positive multiplicative constants) non-zero convex function $u \in C^{1,1}(\overline{\O}) \cap C^{\infty}(\O)$ solving the eigenvalue problem \eqref{eigenvalueproblemforMA}.
\end{Theorem}
The constant $\l_{MA}$ is called the Monge-Amp\`ere eigenvalue and is defined in the following manner. Let $A(x) \in C(\O)$ be a symmetric, positive-definite matrix such that $\det A(x) \geq n^{-n}$ for all $x \in \O$. The collection of all such matrices will be denoted $\mathcal{A}$. Let $L_A$ be the linear operator $L_A v = -\tr(A(x) D^2 v)$, and denote by $\l^1_A$ the (positive) first Dirichlet eigenvalue of $L_A$. Then the Monge-Amp\`ere eigenvalue is defined as 
$$\l_{MA} := \left(\inf_{A \in \mathcal{A}} \lambda^1_A \right)^n.$$
The eigenvalue problem \eqref{eigenvalueproblemforMA} was revisited by Tso in \cite{Tso90} from a variational point-of-view. In order to state Tso's result, we need a few definitions. Consider the class of functions
$$
\K_2 = \lb  u \in C^{0,1}(\overline{\O}) \cap C^{\infty}(\O): \ u \text{ convex and non-zero in } \O, \ u = 0 \text{ on } \partial \O \rb.
$$
Define the Rayleigh quotient of a function $u \in \K_2$ as
$$
R(u) := \frac{\int_{\O} (-u) \ \det D^2u}{\int_{\O} (-u)^{n+1}}.
$$
It is useful to observe that $R(cu) = R(u)$ for all $c > 0$.
\begin{Theorem} [Tso '90]\label{Tso} Assume $\O \subset \R^n$ is a smooth, bounded, uniformly convex domain. Then 
$$\l_{MA} = \inf_{u \in \K_2} R(u).$$
\end{Theorem}
Owing to recent work of Le \cite{Le18}, Theorems \ref{Lions} and \ref{Tso} hold for arbitrary convex domains $\O$, without assuming uniform convexity. To state Le's result, we let
$$
\K = \lb  u \in C(\overline{\O}): \ u \text{ convex and non-zero in } \O, \ u = 0 \text{ on } \partial \O \rb.
$$
Given $u \in \K$, we denote by $Mu$ the Monge-Amp\`ere measure of $u$, defined in \eqref{MAmeasure} in Section \ref{sec:prelim}. The Monge-Amp\`ere energy of $u$ is the quantity $I(u) := \int_{\O} (-u) \ d Mu$. The Rayleigh quotient of $u$ is then defined as 
\begin{equation}\label{RayleighQuotient}
R(u) := \frac{I(u)}{||u||^{n+1}_{L^{n+1}(\O)}} = \frac{\int_{\O} (-u) \ d Mu}{\int_{\O} (-u)^{n+1}}.
\end{equation}
Note that this definition coincides with the one considered by Lions and Tso when $u \in \K_2$.
\begin{Theorem} [Le '18]\label{Le} Assume $\O \subset \R^n$ is a bounded, convex domain.  Then there exists a unique positive constant (still denoted by $\l_{MA}$) and a unique (up to positive multiplicative constants) function $u \in \K \cap C^{\infty}(\O)$ satisfying \eqref{eigenvalueproblemforMA} with
$$\l = \l_{MA} = \inf_{u \in \K} R(u).$$
\end{Theorem}

There are two methods currently available for constructing a solution of \eqref{eigenvalueproblemforMA}, both relying on compactness arguments. The first, by Lions \cite{Lions85}, considers solving the following Dirichlet problem for a convex function $u_{\tau} \in C^2(\overline{\O})$ for each $\tau \geq 0$:
\begin{equation}\label{LionsEquation}
\begin{cases}
\det D^2u_{\tau} = (1 - \tau u_{\tau})^n & \quad \text{in } \O,\\
u_{\tau} = 0 & \quad \text{on } \partial \O.
\end{cases}
\end{equation}
It is shown in \cite[Theorem 1]{Lions85} that the quantity
\begin{equation}\label{definitionofmu}
\mu := \sup\{ \tau > 0 : \text{ there exists a solution } u_{\tau} \text{ of  \eqref{LionsEquation}} \}
\end{equation}
is strictly positive, that $\lim_{\tau \to \mu^{-}} ||u_{\tau}||_{L^{\infty}(\O)} = \infty$, and that (up to choice of a subsequence) the functions $\hat{u}_{\tau} := \frac{u_{\tau}}{||u_{\tau}||_{L^{\infty}(\O)}}$ converge to a solution of \eqref{eigenvalueproblemforMA} as $\tau \to \mu^{-}$. Furthermore, $\mu = \l_{MA}^{\frac{1}{n}}$; thus, \eqref{definitionofmu} provides a third characterization of the Monge-Amp\`ere eigenvalue $\l_{MA}$. 

The second method of constructing a solution of \eqref{eigenvalueproblemforMA}, by Tso \cite{Tso90}, is to fix constants $\sigma, p > 0$ and consider the Dirichlet problem
\begin{equation}\label{TsoEquation}
\begin{cases}
\det D^2u = \sigma (-u)^p & \quad \text{in } \O,\\
u = 0 & \quad \text{on } \partial \O.
\end{cases}
\end{equation}
Notice that the equation \eqref{TsoEquation} is the Euler-Lagrange equation of the functional
\begin{equation}\label{TsoFunctional}
J_{p,\sigma}(u) := \frac{1}{n+1}\int_{\O} (-u) \ \det D^2u - \frac{\sigma}{p+1} \int_{\O} (-u)^{p+1}.
\end{equation}
Using variational methods, Tso proves the existence of unique minimizers in $\mathcal{K}_2$ of the functional $J_{p,\sigma}$ for $p < n$ and $\sigma = \l_{MA}$. By establishing estimates for the minimizers that are uniform in $p$, Tso shows there exists a sequence $p_k \nearrow n$ such that the solutions $u_k$ of \eqref{TsoEquation} with $p = p_k$ and $\sigma = \l_{MA}$ converge to a solution of \eqref{eigenvalueproblemforMA}. 

The primary contribution of the present work is to present an iterative method for constructing a sequence of functions $u_k \in \K$ that converges uniformly to a solution of \eqref{eigenvalueproblemforMA}. This sequence is obtained by repeatedly inverting the Monge-Amp\`ere operator with Dirichlet boundary condition. We show, moreover, that $\lim\limits_{k \to \infty} R(u_k) = \l_{MA}$. Similar inverse iteration methods have been considered for equations in divergence form such as the $p$-Laplace equation \cite{BiezunerErcoleMartins09, HyndLindgren16, Bozorgnia16}. The present work establishes the first inverse iteration result for the eigenvalue problem of a fully nonlinear degenerate elliptic equation.

\begin{Theorem}\label{MainTheorem}
Suppose $\O \subset \Rn$ is a bounded, convex domain. Let $u_0 \in C(\overline{\O})$ satisfy the following conditions:
\begin{itemize}
\item[(i)] $u_0$ is convex and $u_0 \leq 0$ on $\partial \O$,
\item[(ii)] $R(u_0) < \infty$,
\item[(iii)] $Mu_0 \geq \mathcal{L}^n$ in $\O$, where $\mathcal{L}^n$ denotes $n$-dimensional Lebesgue measure.
\end{itemize}
For $k \geq 0$, define the sequence $u_k \in \K$ to be the solutions of the Dirichlet problem
\begin{equation}\label{iteration}
\begin{cases}
\det D^2u_{k+1} = R(u_k) (-u_k)^n & \quad \text{in } \O, \\
u_{k+1} = 0 & \quad \text{on } \partial \O.
\end{cases}
\end{equation}
Then $\{u_k\}$ converges uniformly on $\overline{\O}$ to a non-zero Monge-Amp\`ere eigenfunction $u_{\infty}$. Consequently, the sequence $\hat{u}_k := \frac{u_k}{||u_k||_{L^{\infty}(\O)}}$ converges uniformly on $\overline{\O}$ to the unique solution $u$ of \eqref{eigenvalueproblemforMA} satisfying $||u||_{L^{\infty}(\O)} = 1$. Furthermore, $\lim\limits_{k \to \infty} R(u_k) = \lim\limits_{k \to \infty} R(\hat{u}_k)  = \l_{MA}$.
\end{Theorem}

We briefly outline the strategy behind the proof of Theorem \ref{MainTheorem}. The starting point is a monotonicity relation, proved in Lemma \ref{monotonicitylemma}, which provides control over the Rayleigh quotients $R(u_k)$ and enables us to prove uniform H\"older estimates for the functions $u_k$; see Lemma \ref{HolderEstimate}. The sequence $\{u_k\}$ is, therefore, compact; hence, there exists a subsequence $\{u_{k(j)}\}_{j \in \N}$ converging to a limiting function $u_{\infty}$. Comparison principle arguments using the eigenfunctions from Theorem \ref{Le} show that   $||u_k||_{L^{\infty}(\O)}$ stays uniformly away from zero; see Lemma \ref{nondegeneracy}. Consequently, $u_{\infty} \in \K$ is a candidate to solve the eigenvalue problem \eqref{eigenvalueproblemforMA}. However, in order to prove that $u_{\infty}$ is an eigenfunction, it is necessary to show that the shifted subsequence $\{u_{k(j)+1} \}_{j \in \N}$ also converges to $u_{\infty}$.
The monotonicity relation and a continuity property of the Monge-Amp\`ere energy, Lemma \ref{continuityofenergy}, are essential to verify the aforementioned claim, as well as to establish that any convergent subsequence of $\{u_k\}$ must converge to the same eigenfunction $u_{\infty}$.

Let us point out an elementary construction of an initial function $u_0$ satisfying the hypotheses of Theorem \ref{MainTheorem} for any bounded, convex domain $\O \subset \Rn$. Let $B_R(x_0)$ be any ball centered at $x_0 \in \Rn$ of radius $R > 0$ such that $\O \Subset B_R(x_0)$. Consider the parabola $P_R(x) = \frac{1}{2}\left(|x-x_0|^2 - R^2\right)$, which satisfies $\det D^2 P_R(x) = 1$ for all $x \in \Rn$ and vanishes on $\partial B_R(x_0)$. Then $u_0(x) = P_R(x)$ satisfies all the properties required in the statement of Theorem \ref{MainTheorem}.

We highlight some other noteworthy attributes of the iteration \eqref{iteration}. First, let us point out that both the approaches of Lions and Tso outlined above for constructing a solution of \eqref{eigenvalueproblemforMA} require \emph{a priori} knowledge of the Monge-Amp\`ere eigenvalue $\l_{MA}$. The iterative method \eqref{iteration} solves for both the eigenfunction and eigenvalue simultaneously and thus requires no advance knowledge of $\l_{MA}$. Additionally, \eqref{iteration} provides a means to estimate $\l_{MA}$ by computing the Rayleigh quotients $R(u_k)$ for $k$ large. Approximation of the Monge-Amp\`ere eigenvalue is of interest, as $\l_{MA}$ is known to satisfy analogues of the classical Brunn-Minkowski, isoperimetric, and reverse isoperimetric inequalities; we refer to the works \cite{Salani05, BrandoliniNitschTrombetti09, Hartenstine09, Le18} for the exact statements of these inequalities. It has also been noted in \cite{OlikerGaussCurvatureFlow, LiWangGaussCurvatureFlow} that $\l_{MA}$ should determine the rate of extinction for a class of non-parametric surfaces flowing by the $n$-th root of their Gauss curvature.

Second, the methods of Lions and Tso necessitate solving Dirichlet problems for Monge-Amp\`ere equations of the form $\det D^2 u = f(u)$, where the right-hand side is some function $f$ of the unknown $u$. The iteration \eqref{iteration}, on the other hand, requires solving Dirichlet problems for Monge-Amp\`ere equations of the form $\det D^2 u = g$ where the right-hand side $g$ depends only on the previous iterate, hence is a known function. This makes \eqref{iteration} appealing from the point-of-view of numerical analysis. There is a vast literature on numerical methods for the Dirichlet problem for the Monge-Amp\`ere equation and, more generally, fully nonlinear elliptic equations. We refer the reader to the recent survey \cite{NeilanSalgadoZhangSurvey} for an extensive overview.

Finally, let us recall that the Monge-Amp\`ere operator can also be written in divergence form:
$$\det D^2u = \frac{1}{n} \text{div} (\Phi_u \nabla u),$$
where $\Phi_u (x)$ is the cofactor matrix of $D^2u(x)$, given by $\det D^2u(x)  (D^2u(x))^{-1}$ when $D^2u(x)$ is invertible. An integration by parts shows that one can write the Rayleigh quotient \eqref{RayleighQuotient} in the more familiar manner
$$R(u) = \frac{\frac{1}{n} \int_{\O} \la \Phi_u \nabla u, \nabla u \ra}{\int_{\O} (-u)^{n+1}}.$$
This form of the Rayleigh quotient suggests using appropriate versions of Poincar\'e and Sobolev-type inequalities (see \cite{TianWang08, Maldonado13}) to prove Theorem \ref{MainTheorem}. However, this would require explicit control of the cofactor matrix $\Phi_u$ at each step of the iteration, which is difficult as the smallest eigenvalue of $D^2u$ degenerates near $\partial \Omega$, due to imposing the Dirichlet boundary condition. Our proof of Theorem \ref{MainTheorem} thus relies heavily on techniques for tackling non-divergence form equations and makes full use of various fundamental attributes of convex functions and solutions of the Monge-Amp\`ere equation.

Let us mention that Theorem \ref{MainTheorem} does not provide an independent proof of existence and uniqueness (up to scaling of the eigenfunction) of an eigenpair $(u,\l)$ solving \eqref{eigenvalueproblemforMA}; it merely provides a computational method for obtaining the eigenfunction $u$ of unit height and the eigenvalue $\l_{MA}$. In fact, the proof of Theorem \ref{MainTheorem} uses Theorem \ref{Le}.

The rest of this note is structured as follows: in Section \ref{sec:prelim} we state some basic properties of convex functions and the Monge-Amp\`ere equation. The proof of the main result, Theorem \ref{MainTheorem}, is carried out in Section \ref{sec:mainthm}. 

%%%%%%%%%%%%%%%%%%%%%%%%%%%%%%%%%%%%%%%%
%%%%%%%%%%%%%%%%%%%%%%%%%%%%%%%%%%%%%%%%
%%%%%%%%%%%%%%%%%%%%%%%%%%%%%%%%%%%%%%%%

\section{Background on the Monge-Amp\`ere Equation}\label{sec:prelim}

This section is devoted to stating some basic results on convex functions and weak solutions of the Monge-Amp\`ere equation that will be used in the proof of Theorem \ref{MainTheorem}. From here onward, we will assume that the domain $\O$ is bounded and convex.

Given a function $u \in C(\overline{\O})$, the subdifferential of $u$ at $x \in \O$ is the set
$$\partial u(x) := \{ p \in \Rn : u(y) \geq u(x) + p \cdot (y - x) \text{ for all } y \in \O \}.$$
If $u$ is differentiable at $x$, then $\partial u(x) = \{\nabla u(x)\}$. Given a set $E \subset \O$, we define
$$\partial u(E) := \bigcup_{x \in \O} \partial u(x).$$
The Monge-Amp\`ere measure of $u$ is defined as
\begin{equation}\label{MAmeasure}
Mu(E) := \mathcal{L}^n(\partial u(E)) \quad \text{for all } E \subset \O \text{ such that } \partial u(E) \text{ is Lebesgue measurable,}
\end{equation}
where, $\mathcal{L}^n$ denotes $n$-dimensional Lebesgue measure. It is well known that $Mu$ is a Radon measure (see \cite[Lemma 1.2.2]{GutierrezBook}) and that if $u \in C^2(\O)$,
$$Mu(E) = \int_E \det D^2 u.$$
The following result shows that Monge-Amp\`ere measures are stable under uniform convergence.
\begin{Lemma} [{Weak Convergence of Monge-Amp\`ere Measures; \cite[Lemma 1.2.3]{GutierrezBook} and \cite[Proposition 2.6]{FigalliBook}}]\label{weakconvergence} If $u_k$ are convex functions in $\O$ converging locally uniformly to a function $u$, then the associated Monge-Amp\`ere measures $Mu_k$ converge weakly to the measure $Mu$; that is,
$$\lim\limits_{k \to \infty} \int_{\O} \varphi \ d Mu_k  = \int_{\O} \varphi \ d Mu \quad \text{for all } \varphi \in C_c(\O).$$
\end{Lemma}

Given a non-negative Borel measure $\nu$ on $\O$, we say that the convex function $u \in C(\O)$ is an \emph{Aleksandrov solution} of $\det D^2u = \nu$ in $\O$ if $Mu = \nu$ as measures. We also write $Mu \geq \nu$ in $\O$ (resp. $Mu \leq \nu$ in $\O$) if $Mu(E) \geq \nu(E)$ (resp. $Mu(E) \leq \nu(E)$) for all Borel sets $E \subset \O$. If $\nu$ is absolutely continuous with respect to $n$-dimensional Lebesgue measure and has a density $f$, then we will write $\det D^2 u = f$.

We next state the interior gradient estimate, the Aleksandrov maximum principle, and the comparison principle for Aleksandrov solutions.

\begin{Lemma}[{Interior Gradient Estimate; \cite[Lemma 3.2.1]{GutierrezBook}}]\label{gradientestimate} Suppose $u \in C(\overline{\O})$ is convex and vanishes on $\partial \O$. Then
\begin{equation}
|p| \leq \frac{\sup_{\O} |u|}{\dist(x,\partial \O)} \quad \text{for all } x \in \O, \ p \in \partial u(x).
\end{equation}
\end{Lemma}

\begin{Theorem}[{Aleksandrov Maximum Principle; \cite[Theorem 1.4.2]{GutierrezBook}}]\label{alexandrov}  Suppose $u\in C(\overline{\O})$ is convex and vanishes on $\partial \O$. Then there exists a constant $C_n > 0$ depending only on the dimension $n$ such that
\begin{equation}
|u(x)|^n \leq C_n \text{diam}(\O)^{n-1} \text{dist}(x,\partial \O) Mu(\O) \quad \text{for all } x \in \O.
\end{equation}
\end{Theorem}

\begin{Lemma} [{Comparison Principle; \cite[Theorem 1.4.6]{GutierrezBook}}]\label{comparison} Suppose $u, v \in C(\overline{\O})$ are convex and satisfy $u \geq v$ on $\partial \O$ and $Mu \leq Mv$ in $\O$. Then $u \geq v$ in $\O$.
\end{Lemma}

The following result due to Hartenstine \cite{Hartenstine06} shows that the Dirichlet problem for the Monge-Amp\`ere equation on any bounded, convex domain with zero boundary data always has a unique Aleksandrov solution; see also \cite[Theorem 2.1.3]{FigalliBook}.

\begin{Theorem} [{Solvability of Dirichlet Problem; \cite[Theorem 1]{Hartenstine06}}] Given a Borel measure $\nu$ with $\nu(\O) < \infty$, there exists a unique convex function $u \in C(\overline{\O})$ that is an Aleksandrov solution of the Dirichlet problem
$$
\begin{cases}
\det D^2u = \nu & \quad \text{in } \O, \\
u = 0 & \quad \text{on } \partial \O.
\end{cases}
$$
\end{Theorem}

Aleksandrov solutions of the Dirichlet problem with zero boundary conditions are closed under uniform limits, as shown by the following Lemma.

\begin{Lemma}[{Stability of Aleksandrov Solutions; \cite[Proposition 2.12]{FigalliBook}}] \label{stability} Let $\{\nu_k\}$ be a sequence of Borel measures in $\O$ such that $\sup_k \nu_k(\O) < \infty$ and let $u_k \in C(\overline{\O})$ be Aleksandrov solutions of the Dirichlet problem 
$$
\begin{cases}
\det D^2u_k = \nu_k & \quad \text{in } \O, \\
u_k = 0 & \quad \text{on } \partial \O.
\end{cases}
$$
If $\nu_k$ converges weakly to a Borel measure $\nu$ on $\O$, then $u_k$ converges locally uniformly to the Aleksandrov solution $u$ of the Dirichlet problem
$$
\begin{cases}
\det D^2u = \nu & \quad \text{in } \O, \\
u = 0 & \quad \text{on } \partial \O.
\end{cases}
$$
\end{Lemma}

A hallmark result in the theory of Monge-Amp\`ere equations is the strict convexity and regularity of Aleksandrov solutions established by Caffarelli in the seminal works \cite{CaffarelliW2p, CaffarelliLocalization, CaffarelliRegularity}. We summarize these important contributions as follows.

\begin{Theorem}[{Regularity Results for Aleksandrov solutions; see also \cite[Corollaries 4.11, 4.21, and 4.43]{FigalliBook} and \cite[Theorem 5.4.8]{GutierrezBook}}] \label{regularityofMAeqn} 

Let $u$ be an Aleksandrov solution of the Dirichlet problem
$$
\begin{cases}
\det D^2u = f & \quad \text{in } \O, \\
u = 0 & \quad \text{on } \partial \O.
\end{cases}
$$
Suppose there exist constants $C_1, C_2 > 0$ such that $C_1 \leq f \leq C_2$ in $\O$. Then the following results hold:
\begin{itemize}
\item[(i)] $u$ is strictly convex and $u \in C^{1,\a}_{loc}(\O)$.
\item[(ii)] If $f \in C^{\a}(\O)$, then $u \in C^{2,\a}_{loc}(\O)$.
\item[(iii)] If $f \in C^{\infty}(\O)$, then $u \in C^{\infty}(\O)$.
\end{itemize}
\end{Theorem}

Standard bootstrap arguments using Theorem \ref{regularityofMAeqn} show that Aleksandrov solutions of the Monge-Amp\`ere eigenvalue problem are strictly convex and smooth in the interior (see \cite[Proposition 2.8]{Le18}).

\begin{Proposition} [Interior Regularity]\label{interiorregularity}  Let $\sigma, p > 0$ be fixed constants. Suppose $u \in C(\overline{\O})$ is a non-zero Aleksandrov solution of the Dirichlet problem
$$
\begin{cases}
\det D^2u = \sigma (-u)^p & \quad \text{in } \O, \\
u = 0 & \quad \text{on } \partial \O.
\end{cases}
$$
Then $u$ is strictly convex and $u \in C^{\infty}(\O) \cap C(\overline{\O})$.
\end{Proposition}

We next prove a continuity property of the Monge-Amp\`ere energy, $I(u) = \int_{\O} (-u) d Mu$ along a sequence of convex functions $\{v_k\}$ converging uniformly and satisfying uniform upper bounds on $Mv_k$ with respect to Lebesgue measure (cf. \cite[Proposition 1.1]{Tso90}).

\begin{Lemma}\label{continuityofenergy} Suppose $v_k \in C(\overline{\O})$ are convex functions converging uniformly on $\overline{\O}$ to a function $v$, and there exists a constant $\Lambda > 0$ such that $Mv_k \leq \Lambda \mathcal{L}^n$ for all $k \geq 0$. Then $\lim\limits_{k \to \infty} I(v_k) = I(v)$. \end{Lemma}

\begin{proof} Let $\varphi \in C_c(\O)$ be arbitrary. We have
\begin{align*}
\bigg|  \int_{\O} \varphi v \ dMv - \int_{\O} \varphi v_k \ dMv_k \bigg| 
& \leq \bigg| \int_{\O} \varphi v \ dMv - \int_{\O} \varphi v \ dMv_k  \bigg| + \bigg| \int_{\O} \varphi(v - v_k) \ dMv_k \bigg| \\
& \leq \bigg| \int_{\O} \varphi v \ dMv - \int_{\O} \varphi v \ dMv_k  \bigg| + ||\varphi||_{L^{\infty}(\O)}||v - v_k||_{L^{\infty}(\O)} Mv_k(\O) \\
& \leq  \bigg| \int_{\O} \varphi v \ dMv - \int_{\O} \varphi v \ dMv_k  \bigg| + ||\varphi||_{L^{\infty}(\O)}||v - v_k||_{L^{\infty}(\O)} \Lambda\mathcal{L}^n(\O) \\
& =: A_k + B_k.
\end{align*}
By Lemma \ref{weakconvergence}, we know $\lim_{k \to \infty} A_k = 0$ while $\lim_{k \to \infty} B_k = 0$ due to the uniform convergence of $v_k$ to $v$. Therefore, 
\begin{equation}\label{convergencewithcutoff}
\lim_{k \to \infty} \int_{\O} \varphi v_k \ dMv_k = \int_{\O} \varphi v \ dMv \quad \text{for all } \varphi \in C_c(\O).
\end{equation}
Now let $\e > 0$ be fixed and let $\O_{\e}$ be an open set such that $\O_{\e} \Subset \O$ and $\mathcal{L}^n(\O \setminus \overline{\O_{\e}}) \leq \e$. Let $\psi_{\e} \in C_c(\O)$ be such that $0 \leq \psi_{\e} \leq 1$ in $\O$ and $\psi_{\e} \equiv 1$ on $\overline{\O_{\e}}$. Then, for any $k \geq 0$, we can write
\begin{align*}
I(v_k) - I(v) & = \int_{\O} v \ dMv - \int_{\O} v_k \ dMv_k \\
& = \int_{\O} \psi_{\e} v \ dMv - \int_{\O} \psi_{\e} v_k \ dMv_k+\int_{\O} (1 - \psi_{\e}) v \ dMv - \int_{\O} (1 - \psi_{\e}) v_k \ dMv_k \\
& = \int_{\O} \psi_{\e} v \ dMv - \int_{\O} \psi_{\e} v_k \ dMv_k + \int_{\O  \setminus \overline{\O_{\e}}} (1 - \psi_{\e}) v \ dMv - \int_{\O  \setminus \overline{\O_{\e}}} (1 - \psi_{\e}) v_k \ dMv_k.
\end{align*}
Since $Mv_k \leq  \Lambda \mathcal{L}^n$ for all $k \geq 0$, the lower semicontinuity on open sets of the Monge-Amp\`ere measure under uniform convergence (see \cite[Lemma 1.2.2 (ii)]{GutierrezBook}) implies $Mv(U) \leq \Lambda \mathcal{L}^n(U)$ for any open set $U \subset \O$. Therefore,
$$\bigg| \int_{\O  \setminus\overline{\O_{\e}}} (1 - \psi_{\e}) v \ dMv \bigg| \leq ||1 - \psi_{\e}||_{L^{\infty}(\O)}||v||_{L^{\infty}(\O)} Mv(\O  \setminus \overline{\O_{\e}}) \leq ||v||_{L^{\infty}(\O)} \Lambda \mathcal{L}^n(\O  \setminus \overline{\O_{\e}}) \leq C_1 \e,$$
where $C_1 > 0$ is a constant independent of $\e$. Similarly,
$$\bigg| \int_{\O  \setminus \overline{\O_{\e}}} (1 - \psi_{\e}) v_k \ dMv_k \bigg| \leq ||1 - \psi_{\e}||_{L^{\infty}(\O)}||v_k||_{L^{\infty}(\O)} Mv_k(\O  \setminus \O_{\e}) \leq ||v_k||_{L^{\infty}(\O)} \Lambda \mathcal{L}^n(\O  \setminus\overline{\O_{\e}}) \leq C_2 \e,$$
where $C_2 > 0$ is a constant independent of $\e$ and $k$.
Therefore, there exists a constant $C > 0$ independent of $k$ and $\e$ such that
$$
|I(v_k) - I(v)| \leq \bigg| \int_{\O} \psi_{\e} v \ dMv - \int_{\O} \psi_{\e} v_k \ dMv_k \bigg| + C\e.
$$
Consequently, by \eqref{convergencewithcutoff}, we have
$$\limsup_{k \to \infty} |I(v_k) - I(v)|  \leq C\e.$$
Since $\e > 0$ was arbitrary, we conclude that
$$\lim\limits_{k \to \infty} I(v_k)=I(v).$$ 
\end{proof}

We conclude this section by showing that if $u \in C(\overline{\O})$ is convex and vanishes on $\partial \O$, then all $L^p$ norms of $u$ are comparable.

\begin{Lemma}\label{normequivalence}
If $u \in C(\overline{\O})$ is convex and vanishes on $\partial \O$, then
$$\frac{||u||_{L^{\infty}(\O)}}{n+1} \leq \left( \frac{1}{|\O|} \int_{\O} |u|^p \right)^{\frac{1}{p}} \leq ||u||_{L^{\infty}(\O)} \quad \text{for all } p \geq 1.$$
\end{Lemma}

\begin{proof}
The second inequality is trivial. For the first, we let $K$ be the convex cone with base $\O$, height $-||u||_{L^{\infty}(\O)}$, and vertex at the point where $u$ achieves its minimum. Then $u \leq K \leq 0$ on $\O$ by convexity of $u$. It follows from Jensen's inequality that for any $p \geq 1$,
$$\left( \frac{1}{|\O|} \int_{\O} |u|^p \right)^{\frac{1}{p}}  \geq \frac{1}{|\O|} \int_{\O} |u| \geq \frac{1}{|\O|} \int_{\O} |K| = \frac{||u||_{L^{\infty}(\O)}}{n+1}.$$
\end{proof}

%%%%%%%%%%%%%%%%%%%%%%%%%%%%%%%%%%%%%%%%
%%%%%%%%%%%%%%%%%%%%%%%%%%%%%%%%%%%%%%%%
%%%%%%%%%%%%%%%%%%%%%%%%%%%%%%%%%%%%%%%%

\section{Proof of Theorem \ref{MainTheorem}}\label{sec:mainthm}

In this entire section, $u_k, \ k \geq 0,$ will always denote the functions from the statement of Theorem \ref{MainTheorem}. We begin the proof of Theorem \ref{MainTheorem} by introducing an important monotone decreasing quantity associated to the iteration \eqref{iteration}.

\begin{Lemma}\label{monotonicitylemma}
\begin{equation}\label{monotonicity}
R(u_{k+1}) ||u_{k+1}||^n_{L^{n+1}(\O)} \leq R(u_k) ||u_k||^n_{L^{n+1}(\O)} \quad \text{for all } k \geq 0.
\end{equation}
\end{Lemma}
\begin{proof}
Multiplying \eqref{iteration} by $-u_{k+1}$ and integrating yields
$$\int_{\O} (-u_{k+1}) dMu_{k+1}  = R(u_k)\int_{\O} (-u_{k+1}) (-u_k)^n.$$
Using the definition of $R(u_{k+1})$, we can rewrite the left-hand side to get
$$R(u_{k+1})  ||u_{k+1} ||^{n+1}_{L^{n+1}(\O)}= R(u_k)\int_{\O} (-u_{k+1}) (-u_k)^n.$$
Then by H\"older's inequality
$$\int_{\O} (-u_{k+1}) (-u_k)^n \leq  ||u_{k+1} ||_{L^{n+1}(\O)} ||u_k||^n_{L^{n+1}(\O)},$$
and inequality \eqref{monotonicity} follows after dividing by $||u_{k+1} ||_{L^{n+1}(\O)}$.
\end{proof}

We now use the monotonicity relation \eqref{monotonicity} to prove a global H\"older estimate for the functions $u_k$ solving \eqref{iteration}.

\begin{Proposition}\label{HolderEstimate}
There exists $C = C(n,\O, u_0) > 0$ such that for all $k \geq 1$, $u_k \in C^{0,\frac{1}{n}}(\overline{\O})$ with H\"older norm uniformly bounded by $C$.
\end{Proposition}

\begin{proof}

By Theorem \ref{alexandrov} and \eqref{iteration}, we have for any $k \geq 0$ and $x \in \O$
\begin{align*}
|u_{k+1}(x)|^n & \leq C_n \text{diam}(\O)^{n-1} \text{dist}(x,\partial \O) Mu_{k+1}(\O)\\
& = C_n \text{diam}(\O)^{n-1} \text{dist}(x,\partial \O) R(u_k) \int_{\O} (-u_k)^n \\
& \leq C_n \text{diam}(\O)^{n-1} \text{dist}(x,\partial \O) R(u_k) ||u_k||^n_{L^{n+1}(\O)} |\O|^{\frac{1}{n+1}} \\
& \leq \left(C_n \text{diam}(\O)^{n-1} |\O|^{\frac{1}{n+1}}  R(u_0) ||u_0||^n_{L^{n+1}(\O)} \right) \text{dist}(x,\partial \O) 
\end{align*}
where we have used H\"older's inequality in the third line and the monotonicity relation \eqref{monotonicity} in the final step. In particular, there exists $C_1 = C_1(n,\O,u_0)>0$ such that
$$\sup_{\O} |u_k| \leq C_1.$$
It follows from the interior gradient estimate Lemma \ref{gradientestimate} that $u_k$ is uniformly Lipschitz on any compact subset of $\O$. Next, since $u_k$ vanishes on $\partial \O$, the estimate above yields a uniform $C^{0,\frac{1}{n}}$ estimate of $u_k$ near $\partial \O$. Consequently, $u_k$ is uniformly $\frac{1}{n}$-H\"older continuous in $\overline{\O}$.

\end{proof}

The next proposition establishes a uniform lower bound for $||u_k||_{L^{\infty}(\O)}$.

\begin{Proposition}\label{nondegeneracy}
$||u_k||_{L^{\infty}(\O)} \geq \l_{MA}^{-1/n}$ for all $k \geq 0$.
\end{Proposition}

\begin{proof} Let $\hat{u}\in \K \cap C^\infty(\O)$ be the solution of \eqref{eigenvalueproblemforMA} satisfying $||\hat{u}||^n_{L^{\infty}(\O)} = \l_{MA}^{-1}$, which exists by Theorem \ref{Le}. We prove by induction that $\hat{u} \geq u_k$ for each $k \geq 0$. To establish the base case, we recall that $M u_0 \geq \mathcal{L}^n$. Therefore, if $E\subset \Omega$ is any Borel set,
$$M\hat{u}(E) = \l_{MA} \int_E(-\hat{u})^n \leq \l_{MA}  \l_{MA}^{-1} \mathcal{L}^n(E) \leq Mu_0(E).$$
Since $\hat{u} = 0$ on $\partial \O$ and $u_0 \leq 0$ on $\partial \O$, it follows from the comparison principle Lemma \ref{comparison} that $\hat{u} \geq u_0$ in $\O$.

Now suppose $\hat{u} \geq u_k$ on $\O$ for some $k \geq 0$. Then for any Borel $E\subset \O$, we have by the characterization of $\l_{MA}$ in Theorem \ref{Le}
$$M u_{k+1} (E)= R(u_k) \int_E(-u_k)^n \geq \l_{MA} \int_E(-u_k)^n \geq \l_{MA} \int_E(-\hat{u})^n  = M\hat{u}(E).$$
Since $u_{k+1} = \hat{u} = 0$ on $\partial \O$, it follows from the comparison principle Lemma \ref{comparison} that $\hat{u} \geq u_{k+1}$ in $\O$. 
\end{proof}

Applying Proposition \ref{nondegeneracy} and Lemma \ref{normequivalence} to the monotonicity relation \eqref{monotonicity} provides an upper bound for the Rayleigh quotients $R(u_k)$.

\begin{Corollary}\label{upperboundforRQ}
There exists a positive constant $C$ depending only on $n, \mathcal{L}^n(\O), \l_{MA}$, and $u_0$ such that $R(u_k) \leq C$ for all $k \geq 1.$
\end{Corollary}

We are now ready to prove the main theorem.

\begin{proof}[Proof of Theorem \ref{MainTheorem}]
By Proposition \ref{HolderEstimate}, the sequence $\{u_k\}_{k=1}^{\infty}$ is uniformly bounded and equicontinuous. Consequently, by the Arzel\`a-Ascoli theorem, it is possible to choose a subsequence $\{k(j)\}_{j  \in \N}$ of indices such that $\{u_{k(j)} \}_{j=1}^{\infty}$ converges uniformly on $\overline\O$ to a convex function $u_{\infty} \in C(\overline{\O})$ with $u_\infty\equiv 0$ on $\partial \O$, while the shifted sequence $\{u_{k(j)+1} \}_{j=1}^{\infty}$ converges uniformly on $\overline\O$ to a convex function $w_{\infty} \in C(\overline{\O})$ with $w_\infty\equiv 0$ on $\partial \O$. Proposition \ref{nondegeneracy} implies $u_{\infty}$ and $w_{\infty}$ are not identically zero. Therefore, $u_{\infty}, w_{\infty} \in \K$.

We verify that the corresponding Rayleigh quotients also converge. Indeed, Proposition \ref{HolderEstimate} and Corollary \ref{upperboundforRQ} show that there exists a constant $\Lambda  > 0$ independent of $k$ such that $Mu_k \leq \Lambda \mathcal{L}^n$ in $\O$ for all $k \geq 1$. Therefore we can apply Lemma \ref{continuityofenergy} and Proposition \ref{nondegeneracy} to conclude that $\lim\limits_{j \to \infty} R(u_{k(j)}) = R(u_{\infty})$ and $\lim\limits_{j \to \infty} R(u_{k(j)+1}) = R(w_{\infty})$.

Next, Lemma \ref{weakconvergence} implies the measures $\nu_j:=R(u_{k(j)})(-u_{k(j)})^n\mathcal{L}^n$ converge weakly to the measure $\nu:=R(u_{\infty}) (-u_\infty)^n\mathcal{L}^n$ as $j\to\infty$. Furthermore, Proposition \ref{HolderEstimate} and Corollary \ref{upperboundforRQ} imply $\sup_{j} \nu_j(\O) <\infty$. Since $Mu_{k(j)+1} = \nu_j$ and $u_{k(j)+1}$ converge uniformly to $w_{\infty}$, we may apply Lemma \ref{stability} to conclude that $\det D^2w_\infty=R(u_{\infty}) (-u_\infty)^n$ in the Aleksandrov sense. 

We claim $w_{\infty} = u_{\infty}$. By the monotonicity relation \eqref{monotonicity}, we have
$$R(u_{k(j+1)}) ||u_{k(j+1)}||^n_{L^{n+1}(\O)} \leq R(u_{k(j)+1}) ||u_{k(j)+1}||^n_{L^{n+1}(\O)} \leq R(u_{k(j)}) ||u_{k(j)}||^n_{L^{n+1}(\O)}, \quad j \in \N.$$
Letting $j \to \infty$, we conclude that
\begin{equation}\label{equalityinmonotonicityatthelimit}
R(w_{\infty}) ||w_{\infty}||^n_{L^{n+1}(\O)} = R(u_{\infty}) ||u_{\infty}||^n_{L^{n+1}(\O)}.
\end{equation}
On the other hand, multiplying the equation $\det D^2w_\infty=R(u_{\infty}) (-u_\infty)^n$ by $-w_{\infty}$ and integrating yields
\begin{align*}
R(w_{\infty}) ||w_{\infty}||^{n+1}_{L^{n+1}(\O)} &= \int_{\O} (-w_{\infty}) \ dMw_\infty \\
& = R(u_{\infty}) \int_{\O} (-w_{\infty}) (-u_\infty)^n \\
& \leq  R(u_{\infty}) ||w_{\infty} ||_{L^{n+1}(\O)} ||u_{\infty}||^n_{L^{n+1}(\O)} & \text{by H\"older's inequality} \\
& = R(w_{\infty}) ||w_{\infty}||^{n+1}_{L^{n+1}(\O)} & \text{by \eqref{equalityinmonotonicityatthelimit}}.
\end{align*}
This shows we have equality in H\"older's inequality, and so there exists a constant $c > 0$ such that $(-w_{\infty})^{n+1} = c(-u_{\infty})^{n+1}$. In particular, $R(u_{\infty}) = R(w_{\infty})$. It follows from \eqref{equalityinmonotonicityatthelimit} that $c = 1$, and consequently, $w_{\infty} = u_{\infty}$. Since $\det D^2 u_{\infty} = R(u_{\infty}) (-u_{\infty})^n$ in the Aleksandrov sense, Theorem \ref{Le} implies $u_{\infty}$ is a Monge-Amp\`ere eigenfunction and $R(u_{\infty}) = \l_{MA}$.

We next show that the full sequence $\{u_k\}_{k=1}^\infty$ converges to the same eigenfunction $u_{\infty}$. Indeed, suppose $\{u_{k_1(j)}\}_{j=1}^\infty$ and $\{u_{k_2(j)}\}_{j=1}^\infty$ are two subsequences of $\{u_k\}_{k=1}^\infty$ converging uniformly to $u_{1, \infty}$ and $u_{2, \infty}$ respectively. By the argument outlined in the preceding paragraphs, both $u_{1, \infty}$ and $u_{2, \infty}$ are eigenfunctions and $R(u_{1, \infty}) = R(u_{2, \infty}) = \l_{MA}$. We construct two new subsequences $\{u_{i_1(j)}\}_{j=1}^\infty$ and $\{u_{i_2(j)}\}_{j=1}^\infty$ by setting $i_1(1)=k_1(1)$, then inductively defining
\begin{align*}
 i_2(j)&=\min_l\{k_2(l)\mid k_2(l)> i_1(j)\}, \quad j \geq 1,\\
 i_1(j)&=\min_l\{k_1(l)\mid k_1(l)> i_2(j-1)\}, \quad j \geq 2.
\end{align*}
Clearly $\{u_{i_1(j)}\}_{j=1}^\infty$ and $\{u_{i_2(j)}\}_{j=1}^\infty$ converge uniformly to the original limits $u_{1, \infty}$ and $u_{2, \infty}$ respectively, while $i_1(j)< i_2(j)$ and $i_2(j)< i_1(j+1)$ for all $j$. Thus by repeated application of the monotonicity relation \eqref{monotonicity}, we find 
\begin{align*}
 R(u_{i_2(j)}) ||u_{i_2(j)}||^n_{L^{n+1}(\O)} &\leq R(u_{i_1(j)}) ||u_{i_1(j)}||^n_{L^{n+1}(\O)}\\
 R(u_{i_1(j+1)}) ||u_{i_1(j+1)}||^n_{L^{n+1}(\O)} &\leq R(u_{i_2(j)}) ||u_{i_2(j)}||^n_{L^{n+1}(\O)}.
\end{align*}
Taking $j\to\infty$ in both inequalities above and then dividing by $\l_{MA}$ yields $||u_{1, \infty}||_{L^{n+1}(\O)}=||u_{2, \infty}||_{L^{n+1}(\O)}$. Since both $u_{1, \infty}$ and $u_{2, \infty}$ are eigenfunctions, they must be multiples of each other; this shows they are equal. Since this equality holds for any arbitrary pair of subsequences $\{u_{k_1(j)}\}_{j=1}^\infty$ and $\{u_{k_2(j)}\}_{j=1}^\infty$ of $\{u_k\}_{k=1}^\infty$, the entire sequence $\{u_k\}_{k=1}^\infty$ must converge uniformly to the same eigenfunction $u_\infty$.

Finally, since $||u_k||_{L^{\infty}(\O)}$ is uniformly bounded away from zero by Proposition \ref{nondegeneracy}, we see the sequence $\{\frac{u_k}{||u_k||_{L^{\infty}(\O)}}\}$ converges uniformly to the unique eigenfunction with $L^\infty$ norm equal to $1$, finishing the proof.
\end{proof}

\section*{Acknowledgments}
We would like to thank the anonymous referees for their detailed feedback on a previous version of this manuscript. In particular, we are grateful to them for pointing out the need for a more careful treatment of the subsequential convergence argument in the proof of the main theorem.

\bibliography{references}

\providecommand{\bysame}{\leavevmode\hbox to3em{\hrulefill}\thinspace}
\providecommand{\MR}{\relax\ifhmode\unskip\space\fi MR }
\providecommand{\MRhref}[2]{%
  \href{http://www.ams.org/mathscinet-getitem?mr=#1}{#2}
}
\providecommand{\href}[2]{#2}
\begin{thebibliography}{10}

\bibitem{BiezunerErcoleMartins09}
Rodney~Josu\'{e} Biezuner, Grey Ercole, and Eder~Marinho Martins,
  \emph{Computing the first eigenvalue of the {$p$}-{L}aplacian via the inverse
  power method}, J. Funct. Anal. \textbf{257} (2009), no.~1, 243--270.
  \MR{2523341}

\bibitem{Bozorgnia16}
Farid Bozorgnia, \emph{Convergence of inverse power method for first eigenvalue
  of {$p$}-{L}aplace operator}, Numer. Funct. Anal. Optim. \textbf{37} (2016),
  no.~11, 1378--1384. \MR{3568275}

\bibitem{BrandoliniNitschTrombetti09}
B.~Brandolini, C.~Nitsch, and C.~Trombetti, \emph{New isoperimetric estimates
  for solutions to {M}onge-{A}mp\`ere equations}, Ann. Inst. H. Poincar\'{e}
  Anal. Non Lin\'{e}aire \textbf{26} (2009), no.~4, 1265--1275. \MR{2542724}

\bibitem{CaffarelliW2p}
Luis~A. Caffarelli, \emph{Interior {$W^{2,p}$} estimates for solutions of the
  {M}onge-{A}mp\`ere equation}, Ann. of Math. (2) \textbf{131} (1990), no.~1,
  135--150. \MR{1038360}

\bibitem{CaffarelliLocalization}
\bysame, \emph{A localization property of viscosity solutions to the
  {M}onge-{A}mp\`ere equation and their strict convexity}, Ann. of Math. (2)
  \textbf{131} (1990), no.~1, 129--134. \MR{1038359}

\bibitem{CaffarelliRegularity}
\bysame, \emph{Some regularity properties of solutions of {M}onge {A}mp\`ere
  equation}, Comm. Pure Appl. Math. \textbf{44} (1991), no.~8-9, 965--969.
  \MR{1127042}

\bibitem{FigalliBook}
Alessio Figalli, \emph{The {M}onge-{A}mp\`ere equation and its applications},
  Zurich Lectures in Advanced Mathematics, European Mathematical Society (EMS),
  Z\"{u}rich, 2017. \MR{3617963}

\bibitem{GutierrezBook}
Cristian~E. Guti\'{e}rrez, \emph{The {M}onge-{A}mp\`ere equation}, Progress in
  Nonlinear Differential Equations and their Applications, vol.~89,
  Birkh\"{a}user/Springer, [Cham], 2016, Second edition [of MR1829162].
  \MR{3560611}

\bibitem{Hartenstine06}
David Hartenstine, \emph{The {D}irichlet problem for the {M}onge-{A}mp\`ere
  equation in convex (but not strictly convex) domains}, Electron. J.
  Differential Equations (2006), No. 138, 9. \MR{2276563}

\bibitem{Hartenstine09}
\bysame, \emph{Brunn-{M}inkowski-type inequalities related to the
  {M}onge-{A}mp\`ere equation}, Adv. Nonlinear Stud. \textbf{9} (2009), no.~2,
  277--294. \MR{2503830}

\bibitem{HyndLindgren16}
Ryan Hynd and Erik Lindgren, \emph{Inverse iteration for {$p$}-ground states},
  Proc. Amer. Math. Soc. \textbf{144} (2016), no.~5, 2121--2131. \MR{3460172}

\bibitem{Le18}
Nam~Q. Le, \emph{The eigenvalue problem for the {M}onge-{A}mp\`ere operator on
  general bounded convex domains}, Ann. Sc. Norm. Super. Pisa Cl. Sci. (5)
  \textbf{18} (2018), no.~4, 1519--1559. \MR{3829755}

\bibitem{LiWangGaussCurvatureFlow}
Xiaolong Li and Kui Wang, \emph{Nonparametric hypersurfaces moving by powers of
  {G}auss curvature}, Michigan Math. J. \textbf{66} (2017), no.~4, 675--682.
  \MR{3720319}

\bibitem{Lions85}
P.-L. Lions, \emph{Two remarks on {M}onge-{A}mp\`ere equations}, Ann. Mat. Pura
  Appl. (4) \textbf{142} (1985), 263--275 (1986). \MR{839040}

\bibitem{Maldonado13}
Diego Maldonado, \emph{The {M}onge-{A}mp\`ere quasi-metric structure admits a
  {S}obolev inequality}, Math. Res. Lett. \textbf{20} (2013), no.~3, 527--536.
  \MR{3162845}

\bibitem{NeilanSalgadoZhangSurvey}
Michael Neilan, Abner~J. Salgado, and Wujun Zhang, \emph{Numerical analysis of
  strongly nonlinear {PDE}s}, Acta Numer. \textbf{26} (2017), 137--303.
  \MR{3653852}

\bibitem{OlikerGaussCurvatureFlow}
Vladimir Oliker, \emph{Evolution of nonparametric surfaces with speed depending
  on curvature. {I}. {T}he {G}auss curvature case}, Indiana Univ. Math. J.
  \textbf{40} (1991), no.~1, 237--258. \MR{1101228}

\bibitem{Salani05}
Paolo Salani, \emph{A {B}runn-{M}inkowski inequality for the {M}onge-{A}mp\`ere
  eigenvalue}, Adv. Math. \textbf{194} (2005), no.~1, 67--86. \MR{2141854}

\bibitem{TianWang08}
Gu-Ji Tian and Xu-Jia Wang, \emph{A class of {S}obolev type inequalities},
  Methods Appl. Anal. \textbf{15} (2008), no.~2, 263--276. \MR{2481683}

\bibitem{Tso90}
Kaising Tso, \emph{On a real {M}onge-{A}mp\`ere functional}, Invent. Math.
  \textbf{101} (1990), no.~2, 425--448. \MR{1062970}

\end{thebibliography}
\bibliographystyle{amsplain}

\end{document}